\documentclass[reqno]{amsart}

\usepackage{mathrsfs}
\usepackage{amsthm}
\usepackage{amscd}
\usepackage{amsmath}
\usepackage{amssymb}
\usepackage{extarrows}
\usepackage[colorlinks, linkcolor=blue]{hyperref}

\theoremstyle{definition}
\newtheorem{defn}{Definition}[section]

\theoremstyle{plain}
\newtheorem{thm}{Theorem}[section]
\newtheorem*{thm*}{Theorem}

\newtheorem{prop}[thm]{Proposition}
\newtheorem{lem}[thm]{Lemma}

\newtheorem*{maintheorem}{Main Theorem}
\newtheorem*{mainthm'}{Main Theorem'}

\theoremstyle{remark}
\newcommand{\mbbt}{\mathbb{T}}

\newcommand{\mclc}{\mathcal{C}}

\newcommand{\mcle}{\mathcal{E}}

\newcommand{\mclh}{\mathcal{H}}
\newcommand{\mcli}{\mathcal{I}}

\newcommand{\mclm}{\mathcal{M}}

\newcommand{\mcls}{\mathcal{S}}

\newcommand{\mscp}{\mathscr{P}}

\newcommand{\mscs}{\mathscr{S}}

\newcommand{\hB}{\hat{B}}

\newcommand{\hK}{\hat{K}}

\newcommand{\whT}{\widehat{T}}

\newcommand{\dif}{\:\!\mathrm{d}}

\newcommand{\supp}{\mathrm{supp}}

\DeclareMathOperator{\dist}{dist} %distance
\title[Zero Entropy Optimization]{Low complexity of optimizing measures over an expanding circle map}

\author{Rui Gao}
\address{Rui Gao: College of Mathematics, Sichuan University, Chengdu 610064, China}
\email{gaoruimath@scu.edu.cn}

\author{Weixiao Shen}
\address{Weixiao Shen: Shanghai Center for Mathematical Sciences, Jiangwan Campus, Fudan University, No
2005 Songhu Road, Shanghai 200438, China}
\email{wxshen@fudan.edu.cn}

\begin{document}

\maketitle

\begin{abstract}
In this paper, we prove that for real analytic expanding circle maps, all optimizing measures of a real analytic potential function have zero entropy, unless the potential is cohomologous to constant. We use the group structure of the symbolic space to solve a transversality problem involved. We also discuss applications to optimizing measures for generic smooth potentials and to Lyapunov optimizing measures.
\end{abstract}
\section{Introduction}
Given a continuous map $T: X\to X$ from a compact metric space into itself and a continuous function $f:X\to \mathbb{R}$, the problem of {\em ergodic optimization} looks for maximization/minimization of $\int f \dif\mu$, where $\mu$ runs over
the collection $\mathcal{M}(T)$ of all $T$-invariant Borel probability measures. Let
\begin{equation}\label{eqn:betaalpha}
\beta(T,f)=\beta(f)=\sup_{\mu\in\mathcal{M}(T)} \int f \dif\mu, \quad
\alpha(T,f)=\alpha(f)=\inf_{\mu \in \mathcal{M}(T)}\int f \dif\mu.
\end{equation}
%\begin{equation}\label{eqn:maxm}
%\beta(T,f)=\beta(f)=\sup_{\mu\in\mathcal{M}(T)} \int f \dif\mu,\,\,\,\,
%\alpha(T,f)=\alpha(f)=\inf_{\mu \in \mathcal{M}(T)}\int f \dif\mu.
%\end{equation}
By Birkhoff's Ergodic Theorem, these quantities also optimize the time average of $f$ along orbits.
%$\beta(f)$ can be equivalently expressed as
%\begin{equation} \label{eqn:maxo}
%\beta(f)=\sup_{x\in X} \limsup_{n\to\infty}\frac{1}{n} \sum_{i=0}^{n-1}f(T^i(x)).
%\end{equation}

A measure $\mu\in\mathcal{M}(T)$ is called a {\em maximizing measure} if $\beta(f)=\int f \dif\mu$ and an orbit $\{T^n(x)\}_{n=0}^{\infty}$  is called a {\em maximizing orbit} if $\beta(f)=\lim_n \frac{1}{n} \sum_{i=0}^{n-1} f(T^i(x)).$
Minimizing measures/orbits are similarly defined. As $-\alpha(f)=\beta(-f)$, the problems of maximization and minimization are equivalent. In this paper, we shall  mainly discuss the maximizing problem.

Numerical experiments  by Hunt-Ott in \cite{HO96L, HO96E} indicate that typically maximizing orbits are periodic and maximizing measures are supported on periodic orbits or at least of zero entropy.
The mathematical theory of ergodic optimization has been developed since then, mainly in the case that $T$ is hyperbolic.  Yuan-Hunt~\cite{YH96} dealt with the case that $T$ is either an Axiom A diffeomorphism or a non-invertible uniformly expanding map, and $f$ is Lipschitz, and proved that only periodic orbits (measures) can be persistently maximal in the Lipschitz topology. Building upon  \cite{YH96, BQ07, Mor08,QS12},
Contreras~\cite{Con16} proved that for $T$ uniformly expanding, a generic $f$ has a unique maximizing measure which is supported on a periodic orbit in the Lipschitz topology. Recently, Huang et al~\cite{HLMXZ19} extended the results and proved that when  $T$ is either uniformly expanding or Axiom A, for a generic $f$ in the $C^1$ topology, the maximizing measure is unique and supported on a periodic orbit. However, just as in many other circumstances in the  state-of-the-art theory of dynamical systems, the local perturbation technique is tied to the $C^1$ (or coarser) topology.

The simplest uniformly expanding maps are probably the maps $T: \mathbb{R}/\mathbb{Z}\to \mathbb{R}/\mathbb{Z}$, $x\mapsto dx \mod 1$ with $d\ge 2$. This case has received much attention. When $d=2$, for $f_\theta(x)=\cos (2\pi (x-\theta))$, $\theta\in \mathbb{R}$, Bousch~\cite{Bou00} proved that the maximizing measure of $f_\theta$ is unique and is a Sturmian measure which, in particular, has entropy zero. See~\cite{Jen07, Jen08, Gao21} for similar results in this direction.

We refer to ~\cite{Jen06, Boc18, Jen19} for development on other aspects of ergodic optimization.

The goal of this paper is to prove the following theorem.

\begin{maintheorem}
Let $T:\mathbb{R}/\mathbb{Z}\to \mathbb{R}/\mathbb{Z}$ be a real analytic expanding map and let $f:\mathbb{R}/\mathbb{Z}\to\mathbb{R}$ be a real analytic function. Then
 \begin{itemize}
   \item either $f$ is real analytically cohomologous to constant, i.e. there exists a real analytic function $g$ and a constant $c$ such that $f=g\circ T-g+c$;
   \item or any optimizing measure of $f$ is of zero entropy.
 \end{itemize}
\end{maintheorem}

In particular, the Main Theorem gives a complete solution to a problem raised by O. Jenkinson, see \cite[Problem~3.12]{Jen06}, which was also mentioned in \cite[page 1847]{Boc18}.

Morris~\cite{Mor08} proved that for generic H\"{o}lder or Lipschitz functions, there is a unique maximizing measure and this measure has entropy zero. This result played an important role in the work \cite{Con16} cited above which proves the outstanding TPO conjecture (TPO means ``typically maximizing measures are supported on a periodic orbit'') in the Lipschitz topology. Our Main Theorem implies that Morris' result remains true for $C^r$ functions for any $r\in\{0,1,2,\ldots,\infty\}$, see Theorem~\ref{thm:Cr}. It is our hope that our Main Theorem will throw some insight in proving the TPO conjecture in $C^r$ topology with $r\ge 2$ in the case that the underlying dynamics is an expanding circle map.

%a similar result was proved by Morris in \cite[Theorem~2]{Morris} which
%Another survey on ergodic optimization is \cite{Jen19}.

One of the key ingredients in our proof is a transversality result. To reduce the technicality, we shall only state a special version that is needed for the proof of the Main Theorem here. See Theorem~\ref{thm:C1} for the general statement. Let $T:\mathbb{R}/\mathbb{Z}\to \mathbb{R}/\mathbb{Z}$ and $f:\mathbb{R}/\mathbb{Z}\to \mathbb{R}$ be as in the Main Theorem. We assume that $T$ is orientation preserving and that $T$ fixes $0$.
Let $\whT:\mathbb{R}\to\mathbb{R}$ denote the unique lift of $T$ with $\whT(0)=0$, via the natural projection $\pi:\mathbb{R}\to \mathbb{R}/\mathbb{Z}$. Let $\tau=\whT^{-1}$ and $\tau_i(x)=\tau(x+i)$
for each $i\in \mathbb{Z}$. Let $\widehat{f}=f\circ \pi$.
%:\mathbb{R}\to \mathbb{R}$ be the $1$-periodic function
%Given $\alpha\in (0,1]$, let $f:\mathbb{R}\to\mathbb{C}$ be a 1-periodic and $\alpha$-H\"{o}lder function. Let $\lambda_0=\min_{x\in\mathbb{R}} %|\whT'(x)|$ and fix $\lambda\in \mathbb{C}$ such that
%$$0<|\lambda|<\lambda_0^\alpha.$$
Let $\mathbb{N}=\{1,2,\ldots\}$. For each $\mathbf{i}=(i_n)_{n\ge 1}\in \Sigma_d:=\{0,1,\ldots, d-1\}^{\mathbb{N}},$ consider %\textcolor{red}{(introducing the notations here or referring to \eqref{eqn:hi} (and \eqref{eq:tau_i})???)}
$$h_{\mathbf{i}}(x):= \sum_{n=1}^\infty \left(\widehat{f}\circ \tau_{i_n}\circ \tau_{i_{n-1}}\circ\cdots\circ \tau_{i_1}(x)-\widehat{f}\circ \tau_{i_n}\circ \tau_{i_{n-1}}\circ\cdots\circ \tau_{i_1}(0)\right), \quad x\in\mathbb{R},$$
which is well-defined and real analytic on $\mathbb{R}$.
%Let $\mathbf{0}$ denote the sequence in $\Sigma_d$ whose entries are all $0$, so that
%\[
%h_{\mathbf{0}}=\sum_{n=1}^\infty (f\circ \tau^n -f\circ\tau^n(0)).
%\]
%\[
%h_{\mathbf{0}}=\sum_{n=1}^\infty \lambda^n (f\circ \tau^n -f\circ\tau^n(0)) \quad\text{and hence}\quad f=\lambda^{-1} h_{\mathbf{0}}\circ\whT - h_{\mathbf{0}}+f(0).
%\]

\begin{thm}[Transversality]~\label{thm:transversal}
Under the above circumstances, we have
\begin{itemize}
\item either $f$ is real analytically cohomologous to constant and $h_{\mathbf{i}}\equiv h_{\mathbf{j}}$ holds for all $\mathbf{i},\mathbf{j}\in \Sigma_d$;
\item or $h_{\mathbf{i}}\not\equiv h_{\mathbf{j}}$ holds whenever $\mathbf{i},\mathbf{j}\in \Sigma_d$ and $\mathbf{i}\not=\mathbf{j}$.
\end{itemize}
%Moreover, (1) holds if and only if $h_{\mathbf{0}}$ is $1$-periodic and $f=\lambda^{-1} h_{\mathbf{0}}\circ\whT - h_{\mathbf{0}}+f(0)$.
\end{thm}

We prove this theorem using the odometer structure of the space $\Sigma_d$. An important observation is that $G_0:=\{\mathbf{i}\in \Sigma_d: h_{\mathbf{i}}=h_{\mathbf{0}}\}$ is a closed subgroup of $\Sigma_d$, where $\mathbf{0}$ denotes the element of $\Sigma_d$ whose entries are all $0$. When $d$ is prime, this immediately implies that either $G_0=\{\mathbf{0}\}$ or $G_0=\Sigma_d$. The same conclusion remains true for an arbitrary integer $d\ge 2$ which will be dealt with by Fourier analysis. See Proposition~\ref{prop:C1(i)(iii)}.
Note that this kind of transversality problem appears naturally in the study of skew product over circle expanding maps. In particular, a special case of this type of dichotomy was obtained in \cite[Theorem A]{RS21} using a different method, which represents an important step in the study of a dimension dichotomy for the graphs of Weierstrass-type functions. The method here provides a much more general result with simpler proofs.

%See \S~\ref{sec:tran} and Appendix~\ref{se:app}.

Let us explain how Theorem~\ref{thm:transversal} is used in the proof of the Main Theorem.
Assume that $f$ is not real analytically cohomologous to a constant. Then the second alternative of Theorem~\ref{thm:transversal} holds.  By ergodic decomposition, we only need to show that any ergodic maximizing measure $\mu$ has zero entropy. Let $S$ denote the support of $\mu$,
%\textcolor{red}{ which is a compact $T$-invariant set (could this be deleted? The meaning of ``$T$-invariant" is ambiguous.)}.
We may assume that $T(0)=0$ and $0\not\in S$ so that $S$ can be identified naturally with a compact subset of $(0,1)$.
%Let $\tau_0,\tau_1,\cdots, \tau_{d-1}$ denote the inverse branches of $T:[0,1)\to [0,1)$ and define
The space of inverse orbits in $S$ is naturally identified with
$$\mathcal{S}=\{((i_n)_{n\ge 1}, x): i_n\in \{0,1,\ldots, d-1\}, x\in S, \tau_{i_n}\tau_{i_{n-1}}\cdots \tau_{i_1}(x)\in S, \forall n\}.$$
With the help of a sub-action, we deduce from Theorem~\ref{thm:transversal} that if the section
$$\mathcal{S}_{\mathbf{i}}=\{x\in S: (\mathbf{i}, x)\in \mathcal{S}\}$$
has a limit point $x_0$, then $\mathbf{i}$ belongs to a set which has at most two elements and which is uniquely determined by $x_0$, see
Proposition~\ref{prop:limitpt}. (It is here that we need real analyticity of $f$ and $T$.)
This enables us to apply well-known results in dimension theory of dynamical systems (for example \cite{LY85}) to conclude that $\mu$ has entropy zero.

Now we state a few corollaries of the Main Theorem.

\begin{thm}\label{thm:Cr}
Let $T:\mathbb{R}/\mathbb{Z}\to \mathbb{R}/\mathbb{Z}$ be expanding and real analytic. Let $r\in\{0,1,2,\cdots,\infty\}$. Then for generic
 %\footnote{Here $C^r(\mathbb{R}/\mathbb{Z})$ is considered as a Banach space in the usual way for $r<\infty$, while $C^\infty(\mathbb{R}/\mathbb{Z})$ is only a Frechet space.}
 $f\in C^r(\mathbb{R}/\mathbb{Z})$, any optimizing  measure of $f$ is of zero entropy.
\end{thm}

As mentioned before, Morris \cite{Mor08} proved a similar result for H\"older and Lipschitz functions, using a result of Bressaud and Quas \cite{BQ07}. Our argument is based on our Main Theorem and the fact that $C^r$ functions can be approximated by real analytic ones. %Based on the result of Morris, Contreras \cite{Con16} further proved that optimizing measures are generically supported on a periodic orbit. See also \cite{HLMXZ19} for a different approach of  Contreras' theorem. It might be mentioned that the arguments in \cite{Morris,Con16,HLMXZ19} all rely on a result of Bressaud and Quas \cite{BQ07}, while our argument does not.
It is worthy noting that the $C^0$ case is particularly special, as generically, the unique optimizing measure is fully supported but has zero entropy, see \cite{BJ02, JM08, Bre08}.
%(credited to \cite{JM08}); on the other hand, in \cite[Theorem~C]{BJ02} it was shown that optimizing measures for generic $C^0$ functions are of full support.

%\begin{rmk}
%  Our approach to the result above is based on the fact that real analytic functions are dense in $C^r$. Therefore, it cannot prove similar result %for function spaces such Lipschitz functions or $C^1$ fun
The special case $f=\pm\log|T'|$ is often referred to as {\em Lyapunov optimization}, see \cite{CLT01}. Our result has the following consequences.
 %we have (\textcolor{red}{should some details, say, citing \cite[PROPOSITION 28]{CLT01}, be added?}):

\begin{thm}\label{thm:analytic Lyp}
 Let $T:\mathbb{R}/\mathbb{Z}\to \mathbb{R}/\mathbb{Z}$ be expanding and real analytic. Then
 \begin{itemize}
   \item either there exists a real analytic diffeomorphism $\phi:\mathbb{R}/\mathbb{Z}\to\mathbb{R}/\mathbb{Z}$ such that $\phi\circ T\circ \phi^{-1}(x)=\deg(T)\cdot x$;
   \item or any Lyapunov optimizing measure of $T$ is of zero entropy.
 \end{itemize}
\end{thm}

%The notion of Lyapunov optimizing measure might be first introduced in \cite{CLT01}.

For $r\in\{1,2,3,\cdots,\infty\}$, let $\mathcal{C}^r$ denote the collection of $C^r$ maps from $\mathbb{R}/\mathbb{Z}$ to itself, endowed with the $C^r$ topology. Let $\mcle^r$ be the open subset of $\mathcal{C}^r$ consisting of $C^r$ expanding maps on $\mathbb{R}/\mathbb{Z}$.
 %which is an open subset of $C^r$ self-maps on $\mathbb{R}/\mathbb{Z}$. (\textcolor{red}{should the topology of $\mcle^r$ and/or $C^r(\mathbb{R}/\mathbb{Z})$ %be defined?})

%{\bf Question.} Can we show that  $(T,f)\mapsto \mclh_T(f)$ is upper semi-continuous on $\mcle^1\times C^1(\mathbb{R}/\mathbb{Z})$ (or at least on $\mcle^2\times C^1(\mathbb{R}/\mathbb{Z})$)? It seems that  $(T,f)\mapsto \mclp_T(f)$ is continuous  $\mcle^1\times C^1(\mathbb{R}/\mathbb{Z})$ and it is not hard to prove. What about the continuity of $(T,f)\mapsto h_T(f)$?
%
%If the answer of the first question positive, then we should have:

\begin{thm}\label{thm:Cr Lyp}
 Let $r\in\{1,2,3,\cdots,\infty\}$. For generic $T\in \mcle^r$, any Lyapunov optimizing measure of $T$ is of zero entropy.
\end{thm}

In \cite{HXY21}, it is proved that for generic $T\in\mcle^2$, a Lyapunov optimizing measure is supported on a periodic orbit. See \cite{CLT01,JM08} for earlier results on $C^{1+\alpha}$ and $C^1$ cases respectively.

The rest of the paper is organized as follows. In \S~\ref{sec:tran}, we prove Theorem~\ref{thm:transversal} and a more general transversality result; in \S~\ref{sec:pfmainthm}, we prove the Main Theorem; in \S~\ref{sec:coro}, we prove Theorems~\ref{thm:Cr},~\ref{thm:analytic Lyp} and~\ref{thm:Cr Lyp}.
%In Appendix~\ref{se:app}, we state and prove a more general version of Theorem~\ref{thm:transversal}.

{\bf Acknowledgment.} RG would like to thank Bing Gao and Zeng Lian for helpful discussions. We would also like to thank Jairo Bochi for reading a first version of the paper and several useful remarks. WS is supported by NSFC grant No. 11731003.
\section{A transversality result}
\label{sec:tran}
In this section, we state and prove a more general version of Theorem~\ref{thm:transversal}. The proof exploits the structure of $\Sigma_d=\{0,1,\ldots, d-1\}^\mathbb{N}$ as a compact abelian group.

%In \S~\ref{subsec:closedsug}, we recall some well known facts and determine the structure of closed subgroups of $\Sigma_d$, see %Proposition~\ref{prop:struc}. In \S~\ref{subsec:trans}, we prove a slightly more general form (Proposition~\ref{prop:C1}) of %Theorem~\ref{thm:transversal}.

\subsection{Statement of result}\label{subsec:trans}
We start with the statement of the general transversality result. Let $T:\mathbb{R}/\mathbb{Z}\to \mathbb{R}/\mathbb{Z}$ be a $C^1$ orientation-preserving expanding map with $T(0)=0$ and let $d\ge 2$ be the degree. Let $f:\mathbb{R}/\mathbb{Z}\to {\mathbb{C}}$
be a $C^\alpha$ function for some $\alpha\in (0,1)$. As before, denote by $\whT:\mathbb{R}\to\mathbb{R}$ the lift of $T$ with $\whT(0)=0$, $\tau=\whT^{-1},$ and  $\tau_i(x)=\tau(x+i)$.
Given $\mathbf{i}=(i_1,\cdots,i_n,\cdots)\in \Sigma_d$, for any $n\ge 1$, denote:
\[
\tau_{\mathbf{i},n}:=\tau_{i_n}\circ \cdots \circ\tau_{i_2}\circ \tau_{i_1}.
\]
Let $\lambda_0=\sup_{n=1}^\infty \min_{x\in\mathbb{R}} |(\whT^n)'(x)|^{1/n}>1$.~\footnote{By~\cite[Theorem A.3]{Morris2013}, $\log\lambda_0=\alpha(T, \log |T'|)$, where $\alpha$ is as in (\ref{eqn:betaalpha}).} Note that for any $0<\lambda_1<\lambda_0$, there exists $C\ge 1$ such that
$$\sup_{x\in \mathbb{R}, \mathbf{i}\in \Sigma_d} |\tau'_{\mathbf{i},n}(x)|\le C \lambda_1^{-n}, \quad \forall~n\ge 0.$$
Fix a $\lambda\in \mathbb{C}$ such that $0<|\lambda|<\lambda_0^\alpha$.
Then,
%\textcolor{red}{(is it necessary to add some explanation for the convergence below??? for example: by definition, $|\lambda|<\min_{x\in\mathbb{R}} |(\whT^n)'(x)|^{\alpha/n}$)
%So for any sufficiently large $n$)}
for each $\mathbf{i}\in \Sigma_d$,
\[
h_{\mathbf{i}}(x):=\sum_{n=1}^\infty \lambda^n\left(\widehat{f}\circ \tau_{\mathbf{i},n}(x)-\widehat{f}\circ \tau_{\mathbf{i},n}(0)\right), \quad\quad x\in\mathbb{R},
\]
%\begin{equation}\label{eqn:hi}
%h_{\mathbf{i}}(x):=\sum_{n=1}^\infty \lambda^n\left(f\circ \tau_{\mathbf{i},n}(x)-f\circ \tau_{\mathbf{i},n}(0)\right), \quad\quad x\in\mathbb{R},
%\end{equation}
is a well-defined continuous function from $\mathbb{R}$ to $\mathbb{C}$, where $\widehat{f}=f\circ \pi$ and $\pi:\mathbb{R}\to\mathbb{R}/\mathbb{Z}$ is the natural projection. The following theorem is our main transversality result.
\begin{thm}\label{thm:C1}
Under the circumstances above, the following are equivalent:
\begin{enumerate}
  \item [(i)] there exist $\mathbf{i}\ne\mathbf{j}$ with $h_{\mathbf{i}} \equiv h_{\mathbf{j}}$;
  \item [(ii)] $h_{\mathbf{i}}=h_{\mathbf{0}}$ for any $\mathbf{i}$;
  \item [(iii)] $h_{\mathbf{0}}$ is $1$-periodic;
  \item [(iv)] $h_{\mathbf{0}}$ is $1$-periodic and $\widehat{f}=\lambda^{-1}\cdot h_{\mathbf{0}}\circ \whT-h_{\mathbf{0}} +\widehat{f}(0)$;
  \item [(v)] there exists a $C^\alpha$ function $\phi:\mathbb{R}/\mathbb{Z}\to\mathbb{C}$ such that $f=\lambda^{-1}\cdot\phi\circ T-\phi+const$.
%$$f=\lambda^{-1}\cdot\phi\circ \whT-\phi+const.$$
\end{enumerate}
\end{thm}

\begin{proof}[Proof of Theorem~\ref{thm:transversal} assuming Theorem~\ref{thm:C1}] Assume that the second alternative in the conclusion of Theorem~\ref{thm:transversal} does not hold. Then (i) and hence (ii)-(v) all hold in Theorem~\ref{thm:C1} (for $\lambda=1$). By (iv), $f$ is real analytically cohomologous to constant since $h_{\mathbf{0}}$ is real analytic. By (ii), $h_{\textbf{i}}=h_{\textbf{j}}$ for all $\mathbf{i},\mathbf{j}\in \Sigma_d$. Thus the first alternative in the conclusion of Theorem~\ref{thm:transversal} holds.
\end{proof}
\subsection{Group structure of \texorpdfstring{$\Sigma_d$}{the symbolic space}}\label{subsec:closedsug}
Let us recall the well-known structure of $\Sigma_d$ as a compact abelian group.
%The discussion in this subsection is well known. See, for example, \cite[\S~10]{HR79}.
The topology on $\Sigma_d$ is the product topology of the discrete topology on $\{0,1,\ldots, d-1\}$.
%(\textcolor{red}{a compatible?}) distance function $\rho$ (\textcolor{red}{It seems that this distance function is never used. Is it necessary to %introduce it?}) defined as
%$$\rho(\mathbf{i},\mathbf{j})=\sum_{n=1}^\infty \frac{|i_n-j_n|}{d^n}.$$
The space $\Sigma_d$ becomes a compact abelian group with addition defined as follows. %\marginpar{``follow" is changed to ``follows".}
For $\mathbf{i}=(i_n)_{n\ge 1},\mathbf{j}=(j_n)_{n\ge 1}\in\Sigma_d$, $\mathbf{i}+\mathbf{j}\in\Sigma_d$ is the unique element $\mathbf{k}=(k_n)_{n\ge 1}\in \Sigma_d$ such that for each $n=1,2,\ldots$,
\[
k_1+k_2d+\cdots +k_{n} d^{n-1} = (i_1+j_1)+(i_2+j_2) d+ \cdots +(i_{n}+j_n) d^{n-1} \mod d^n.
\]
Note that the semi-group $\mathbb{Z}_+=\{0,1,2,\cdots\}$ can be naturally embedded into  $\Sigma_d$ as a dense sub-semi-group in the following way:
\[
m=i_1+i_2d+\cdots + i_nd^{n-1} \mapsto (i_1,\cdots, i_n,0,0,\cdots)=:\iota(m).
\]
In particular, $\mathbf{0}=\iota(0)$ is the zero element of $\Sigma_d$. The map $\mathbf{i}\mapsto \mathbf{i} +\iota(1)$ is usually called a ($d$-adic) {\em adding machine}.

The following lemma is a simple observation which plays an important role in our proof.
\begin{lem}\label{lem:add}
Given $\mathbf{i},\mathbf{j},\mathbf{k}\in\Sigma_d$, $h_{\mathbf{i}}=h_{\mathbf{j}}$ if and only if $h_{\mathbf{i}\mathbf{+}\mathbf{k}}=h_{\mathbf{j}+\mathbf{k}}$ . As a result, $G_0:=\{\mathbf{i}:h_{\mathbf{i}}=h_{\mathbf{0}}\}$ is a closed subgroup of $\Sigma_d$.
\end{lem}

\begin{proof}
Since $\Sigma_d$ is a group, we only need to prove the ``only if'' part.  By definition,
%we have %(\textcolor{red}{recall $\iota$?}):
%\[
%\tau_{\mathbf{i},n}(x+1)=\tau_{\mathbf{i}+\iota(1),n}(x) \mod 1, \quad\forall~\mathbf{i}, ~\forall~n \ge 1~,\forall~x\in\mathbb{R}.
%\
for each positive integer $m$,
%(\textcolor{red}{To me the displayed line above is not much more evident than the line below. Might the line above be removed?})
\[
\tau_{\mathbf{i},n}(x+m)=\tau_{\mathbf{i}+\iota(m),n}(x) \mod 1, \quad\forall~\mathbf{i}, ~\forall~n \ge 1~,\forall~x\in\mathbb{R}.
\]
Since $\widehat{f}$ is of period $1$, it follows that
%\[
%\tau_{\imath,n}'(x+\kappa_n)\cdot f'(\tau_{\imath,n}(x+\kappa_n)) = \tau_{\imath+\kappa,n}'(x)\cdot f'(\tau_{\imath+\kappa,n}x) %\quad\forall~\imath,\kappa,~\forall~n \ge 1~,\forall~x,y\in\mathbb{R}.
%\]
for $m\ge 1$, we have:
\[
h_{\mathbf{i}}(x+m)=h_{\mathbf{i}+\iota(m)}(x).
\]
In particular,
\[
h_{\mathbf{i}}=h_{\mathbf{j}}\Rightarrow h_{\mathbf{i}+\iota(m)}=h_{\mathbf{j}+\iota(m)}.
\]
Since $\iota(\mathbb{Z}_+)$ is dense in $\Sigma_d$ and since $\mathbf{k}\mapsto h_{\mathbf{k}}(x)$ is continuous for any fixed $x\in\mathbb{R}$, it follows that
$h_{\mathbf{i}}=h_{\mathbf{j}}$ implies that $h_{\mathbf{i}+\mathbf{k}}=h_{\mathbf{j}+\mathbf{k}}$. %The other direction follows similarly.

In particular, $G_0$ is a subgroup of $\Sigma_d$. Its closedness follows from the fact that $\mathbf{i}\mapsto h_{\mathbf{i}}(x)$ is continuous for each $x$.
\end{proof}

%On the other hand, we have:
%\iffalse
%\begin{lem}\label{lem:period}
%  If $h$ is $qd^{n}$-periodic for some $q\mid d$ and some $n\ge 0$, then $h$ is $1$-periodic.
%\end{lem}
%
%\begin{proof} From the definition of $h=h_{\mathbf{0}}$, we obtain
%\begin{equation}\label{eq:co-homo}
%  \whT'\cdot h\circ \whT=\sum_{k=1}^\infty \Big((f\circ \tau^{k})\circ \whT\Big)'=f'+h.
%\end{equation}
%Since both $f'$ and $\whT'$ are $1$-periodic and since $h(\whT(x+r))=h(\whT(x)+rd)$ for any $r\in\mbbn$, the equality above implies that: if $h$ is $p$-periodic and if $p\mid rd$, then $h$ is also $r$-periodic. The conclusion follows.
%\end{proof}
%\fi

%We shall prove the following proposition in the next subsection using Fourier analysis.

\begin{prop}\label{prop:C1(i)(iii)}
Under the circumstances of Theorem~\ref{thm:C1}, (i) implies (iii).
\end{prop}

The proof of this proposition will be given in the next subsection, using Fourier analysis. Here we give a short proof in the case that $d$ is a prime.
%\marginpar{The red-colored sentences in proof above looks unclear to me. I can see why such $q_m$ exists by choosing a generator (whose existence %might need explanation) $i_1i_2\cdots $ of $G_0$ and then letting $q_m=i_1+i_2d+\cdots+i_m d^{m-1}$; however, this does not guarantee that $q_1=1$. %A proposed revision of the proof is as below.}
%\begin{proof}[Proof of Proposition~\ref{prop:C1(i)(iii)} assuming that $d$ is a prime] We shall prove that
%%$h_{\mathbf{0}}$ is $1$-periodic.
%By Lemma~\ref{lem:add}, $G_0$ is a non-trivial closed subgroup of $\Sigma_d$. For each $m\ge 1$,
%$$A_m:=\{i_1+i_2d+\cdots+i_m d^{m-1} \mod d^m: (i_n)_{n\ge 1}\in G_0\}$$
%is a subgroup of the cyclic group $\mathbb{Z}/(d^m \mathbb{Z})$. \textcolor{red}{So there exists $q_m\in \{0,1,\ldots, d^m-1\}$ such that $A_m$ is %generated by $q_m \mod d^m$ and $q_{m+1}=q_m \mod d^m$}. The group $G_0$ also has the following property:
%$$0i_1i_2\cdots\in G_0\implies i_1i_2\cdots\in G_0.$$
%Thus $q_m\not=0$ for all $m$. \textcolor{red}{Since $d$ is a prime, $q_1=1$}. From $q_{m+1}=q_m\mod d^m$, we obtain by induction that $q_m=1$ for %all $m$. Since $G_0$ is closed, it follows that $G_0=\Sigma_d$.
%%there exists an integer $k\ge 1$, such that
%$$G_0=\{(i_n)_{n=1}^\infty\in \Sigma_d: i_n=0 \mbox{ for all } n<k\}.$$
%On the other hand, if $0\mathbf{i}\in G_0$, then $\mathbf{i}\in G_0$. Thus $k=1$, i.e., $G_0=\Sigma_d$.
%\end{proof}

\begin{proof}[Proof of Proposition~\ref{prop:C1(i)(iii)} assuming that $d$ is a prime]  It suffices to prove that $\iota(1)\in G_0=\{\mathbf{i}\in \Sigma_d: h_{\mathbf{i}}=h_{\mathbf{0}}\}$, because $h_{\iota(1)}=h_{\mathbf{0}}$ means that $h_{\mathbf{0}}$ is $1$-periodic.
Since (i) holds, by Lemma~\ref{lem:add}, $G_0$ is a non-trivial closed subgroup of $\Sigma_d$. For each $m\ge 1$,
$$A_m:=\{i_1+i_2d+\cdots+i_m d^{m-1} \mod d^m: (i_n)_{n\ge 1}\in G_0\}$$
is a subgroup of the cyclic group $\mathbb{Z}/(d^m \mathbb{Z})$. Since the group $G_0$ also has the following property:
$$0i_1i_2\cdots\in G_0\implies i_1i_2\cdots\in G_0,$$
$A_m$ is non-trivial for each $m$. Since $d$ is prime, there exists a unique $k_m\in \{0,1,\ldots, m-1\}$ such that
$A_m$ is generated by $d^{k_m} \mod d^m$.  In particular, $k_1=0$. Since $j\mod d^{m+1}\mapsto j\mod d^m$ induces a surjective homomorphism from $A_{m+1}$ to $A_m$, $d^{k_{m+1}}\mod d^m$ is also a generator of $A_m$, and thus $k_{m+1}=k_m$. In conlcusion, $k_m=0$ for all $m\ge 1$. Since $G_0$ is closed, it follows that $\iota(1)\in G_0$.
\end{proof}

\begin{proof}[Proof of Theorem~\ref{thm:C1}]
(ii) $\implies$ (i). This is trivial.

(i) $\implies$ (iii). This is Proposition~\ref{prop:C1(i)(iii)}.

%By Lemma~\ref{lem:add}, $G_0=\{\mathbf{i}: h_{\mathbf{i}}=h\}$ is non-trivial closed subgroup of $\Sigma_d$. By Proposition~\ref{prop:struc}, $G_0$ %is generated by $\iota(qd^{n})$ for some $q \mid d$ and $n\ge 0$. In particular, $h(x)=h_{\iota(qd^n)}(x)=h(x+qd^n)$, i.e., $h$ is $qd^n$-periodic. %By Lemma~\ref{lem:period}, $h$ is $1$-periodic.
%% By follows from the three lemmas above.

(iii) $\implies $ (iv).  Let $\widehat{f}_0=\widehat{f}-\widehat{f}(0)$. By definition,
\[
h_{\mathbf{0}}=\sum_{n=1}^\infty \lambda^n\cdot \widehat{f}_0\circ \tau^n.
\]
Then
\[
h_{\mathbf{0}}\circ\whT = \sum_{n=1}^\infty \lambda^n \cdot \widehat{f}_0\circ \tau^{n-1} = \lambda(\widehat{f}_0 + h_{\mathbf{0}}).
\]

(iv) $\implies$ (v). This is trivial.
%Thus (iv) holds with $\phi=h$.
%Let $H(x)=\int_0^x h(t) dt$. Integrating both sides of (\ref{eq:co-homo}), we obtain
%$$H\circ \whT =f+ H -f(0).$$
%Let us show that $H$ is $1$-periodic. Since $h$ is $1$-periodic, it suffices to show that $H(1)=0$. By the last displayed equation,
%$H\circ\whT -H$ is $1$-periodic, so $H(\whT(1))-H(1)=H(\whT(0))-H(0)=0$. Since $H(\whT(1))=H(d)=dH(1)$, it follows that $H(1)=0$.

%{\color{red} I suggest ``(iii) $\implies $ (iv)" be shortened as something like this:  Let $\phi(x)=\int_0^x h(t) dt$. Integrating both sides of (\ref{eq:co-homo}) on $[0,x]$ and $[0,1]$ respectively, we obtain
%$$\phi\circ \whT =f+ \phi -f(0) \quad\text{and}\quad \int_0^d h(t)\dif t = \int_0^1 h(t)\dif t=\phi(1).$$
%The conclusion follows.}

(v) $\implies$ (ii). Denote $\tau_{\mathbf{i}, 0}=\mathrm{id}$ and $\widehat{\phi}=\phi\circ\pi$. There exists $c\in\mathbb{C}$ such that for any $\mathbf{i}$ and any $n\ge 1$, %\marginpar{The red-colored phrase is added. $\phi$ are changed to $\widehat{\phi}$ accordingly below.}
%\[
%\widehat{f}\circ \tau_{\mathbf{i},n}= \lambda^{-1}\phi\circ \widehat{T}\circ\tau_{\mathbf{i}, n} -\phi\circ \tau_{\mathbf{i},n}+ c = \lambda^{-1}\phi\circ \tau_{\mathbf{i}, n-1} -\phi\circ \tau_{\mathbf{i},n}+ c.
%\]
\[
\widehat{f}\circ \tau_{\mathbf{i},n}= \lambda^{-1} \widehat{\phi} \circ \widehat{T}\circ\tau_{\mathbf{i}, n} - \widehat{\phi}\circ \tau_{\mathbf{i},n}+ c = \lambda^{-1} \widehat{\phi} \circ \tau_{\mathbf{i}, n-1} - \widehat{\phi} \circ \tau_{\mathbf{i},n}+ c.
\]
It can be rewritten as
%\[
%\widehat{f}\circ \tau_{\mathbf{i},n} - \widehat{f}\circ \tau_{\mathbf{i},n}(0)=\lambda^{-1}(\phi\circ \tau_{\mathbf{i}, n-1}-\phi\circ \tau_{\mathbf{i}, n-1}(0))- (\phi\circ \tau_{\mathbf{i},n}-\phi\circ \tau_{\mathbf{i},n}(0)).
%\]
\[
\widehat{f}\circ \tau_{\mathbf{i},n} - \widehat{f}\circ \tau_{\mathbf{i},n}(0)=\lambda^{-1}(\widehat{\phi}\circ \tau_{\mathbf{i}, n-1} - \widehat{\phi}\circ \tau_{\mathbf{i}, n-1}(0))- (\widehat{\phi}\circ \tau_{\mathbf{i},n} - \widehat{\phi}\circ \tau_{\mathbf{i},n}(0)).
\]
It follows that
$$h_{\mathbf{i}}= \widehat{\phi} - \widehat{\phi}(0)$$
does not depend on $\mathbf{i}$.
%$$(f\circ \tau_{\mathbf{i},n})'= (\phi\circ \tau_{\mathbf{i}, n-1})' -(\phi\circ \tau_{\mathbf{i},n})'.$$
%Therefore,
%$h_{\mathbf{i}}=\phi'$ is independent of $\mathbf{i}$.
%
%{\color{red} I suggest ``(iv) $\implies $ (ii)" be shortened as follows:  For any $\mathbf{i}$ and any $n\ge 1$,
%$$(f\circ \tau_{\mathbf{i,n}})'= (\phi\circ \widehat{T}\circ\tau_{\mathbf{i}, n} -\phi\circ \tau_{\mathbf{i},n})'= (\phi\circ \tau_{\mathbf{i}, %n-1})' -(\phi\circ \tau_{\mathbf{i},n})',$$
%where $\tau_{\mathbf{i}, 0}=id$. Therefore,
%$h_{\mathbf{i}}=\phi'$ is independent of $\mathbf{i}$.}

% (ii) $\iff$ (iii) follows from the first two lemmas. (iii) $\iff$ (iv) follows easily from \eqref{eq:co-homo} (in (iv) $\phi'=h$).
\end{proof}

\subsection{Fourier analysis}
The goal of this subsection is to prove Proposition~\ref{prop:C1(i)(iii)}.

We first describe a procedure which reduces the problem to the case that $T$ is linear. It is well-known that $T$ is topologically conjugate to the linear map $\mathbf{m}_d: x\mapsto dx \mod 1$ via a homeomorphism $\theta:\mathbb{R}/\mathbb{Z}\to \mathbb{R}/\mathbb{Z}$ with $\theta(0)=0$; in particular,  $\theta\circ T= \mathbf{m}_d\circ\theta$. Let $\Theta:\mathbb{R}\to\mathbb{R}$ be the lift of $\theta$ with $\Theta(0)=0$. Then
\[
\Theta\circ \widehat{T}= d\cdot\Theta,
\]
which implies that
$$\Theta\circ \tau_{\mathbf{i},n} \circ \Theta^{-1}(x)= \frac{x+i_1+i_2 d+ \cdots +i_n d^{n-1}}{d^n}=:\xi_n(\mathbf{i},x).$$
Let $F=\widehat{f}\circ \Theta^{-1}$ and let
\[
H_{\mathbf{i}}(x):=h_{\mathbf{i}}\circ\Theta^{-1}(x)=\sum_{m=1}^\infty \lambda^m [F(\xi_m(\mathbf{i},x))-F(\xi_{m}(\mathbf{i},0))].
\]
By definition, we have the following.
\begin{itemize}
  \item $F$ is $1$-periodic.
  \item $h_{\mathbf{i}}$  is $1$-periodic  if and only if $H_{\mathbf{i}}$  is $1$-periodic.
  \item $h_{\mathbf{i}}\equiv h_{\mathbf{j}}$ if and only if  $H_{\mathbf{i}}\equiv H_{\mathbf{j}}$.
  \item If $T=\mathbf{m}_d$, then $\Theta=\mathrm{id}$ and $h_{\mathbf{i}}=H_{\mathbf{i}}$.
\end{itemize}

For each $s\in \mathbb{Z}$, consider
$$\mathcal{H}^s: \Sigma_d\times[0,1)\to\mathbb{C}, \quad (\mathbf{k},x)\mapsto H_{\mathbf{k}}(x+s)-H_{\mathbf{k}}(x).$$
To prove Proposition~\ref{prop:C1(i)(iii)}, we shall use Fourier analysis to show that $\mathcal{H}^s$ is constant for each $s\in\mathbb{Z}$.

The space $\Sigma_d\times[0,1) $ carries a unique Borel probability measure $\mu$ such that for any $m\ge 1$ and any Borel set $U\subset [0,1)$,
$$\mu(\{(\mathbf{k},x)\in \Sigma_d\times[0,1) : \xi_m(\mathbf{k},x)\in U\})=|U|,$$
where $|U|$ denote the standard Lebesgue measure of $U$. Moreover, there is a $\mu$-preserving Borel measurable bijection $m_d: \Sigma_d\times[0,1) \to \Sigma_d\times[0,1) $ defined as
\[
m_d(\mathbf{k},x)= (k_0\mathbf{k},dx-k_0),
\]
where $k_0=\lfloor dx\rfloor$ is the largest integer which is not greater than $dx$, and $k_0\mathbf{k}=(k_n)_{n\ge 0}$ for $\mathbf{k}=(k_n)_{n\ge 1}$.

%Then $m_d$ is a Borel measurable bijiection
%\begin{equation}\label{eqn:mdmu}
%m_d(\mu)=\mu.
%\end{equation}
Indeed, $\xi:(\mathbf{k},x)\mapsto (\xi_m(\mathbf{k},x))_{m\ge 0}$ provides a natural identification between $\Sigma_d\times[0,1) $ and the
space $X(d)$ of backward orbits of $\mathbf{m}_d$:
$$X(d)=\{(x_m)_{m=0}^\infty: x_m\in \mathbb{R}/\mathbb{Z}, \mathbf{m}_d(x_{m+1})=x_m, \forall m\ge 0\}.$$
The map $m_d$ corresponds to the homeomorphism $(x_m)_{m\ge 0}\mapsto (\mathbf{m}_d x_m)_{m\ge 0}$ and $\mu$ corresponds to the lift of the $\mathbf{m}_d$-invariant Lebesgue measure on $\mathbb{R}/\mathbb{Z}$ to the space $X(d)$.
Let us note that the measure $\mu$ corresponds to the Haar measure on the compact abelian group $X(d)$, although we do not need this fact explicitly.

\iffalse
\textcolor{red}{(The basic fact below should be well-known.???)}
\begin{lem}\label{lem:basis}
Let $E_{m,q}(\mathbf{k},x)=e^{2\pi i \xi_m(\mathbf{k},x) q}$ for $m\ge 0$, $q\in\mathbb{Z}$, where $\xi_0(\mathbf{k},x)=x$. Then
\[
\mathcal{E}:=\{E_{m,q}: m\ge 0,  q\in \mathbb{Z}, d\nmid q\}\cup\{1\}
\]
is an orthonormal basis for $L^2(\Sigma_d\times[0,1) , \mu)$. Moreover, for any $n\ge 0$, $\mathcal{E}=\{E_{m,q}: q\in\mathbb{Z}, m\ge n\}$.
\end{lem}
\begin{proof} Clearly, for each $E_{m,q}\in \mathcal{E}$, $\int |E_{m,q}|^2 d\mu=1$. For $E_{m_1,q_1}, E_{m_2,q_2}\in \mathcal{E}$ with $(m_1,q_1)\not=(m_2,q_2)$, let us show
$$\int_{\Sigma_d\times[0,1) } E_{m_1,q_1} \overline{E_{m_2,q_2}} d\mu=0.$$
To this end, assume without loss of generality $m_1\ge m_2$. Then $q_1- d^{m_1-m_2}q_2\not=0$. Since $\xi_{m_2}=d^{m_1-m_2} \xi_{m_1} \mod 1$ and $(\xi_{m_1})_* \mu= dx$,
\begin{align*}
\int_{\Sigma_d\times[0,1) } E_{m_1, q_1} \overline{E_{m_2,q_2}} d\mu & =\int_{\Sigma_d\times[0,1) } e^{2\pi i(\xi_{m_1} q_1-\xi_{m_2} q_2)} d\mu\\
& =\int_{\Sigma_d\times[0,1) } e^{2\pi i \xi_{m_1} (q_1 -d^{m_1-m_2} q_2)}d\mu\\
& =\int_0^1 e^{2\pi i x(q_1-d^{m_1-m_2} q_2) } d(\xi_{m_1})_* \mu(x)\\
& =\int_0^1 e^{2\pi i x(q_1-d^{m_1-m_2} q_2)} dx =0.
\end{align*}
\textcolor{red}{(The paragraph above was NOT checked. The paragraph below was checked.)}

It remains to show that any function $\varphi$ in $L^2(\Sigma_d\times[0,1) , \mu)$ can be (suitably) approximated by linear combinations of functions in $\mathcal{E}$.
\end{proof}
\fi

The following lemma describes symmetry of the functions $\mathcal{H}^s$.
\begin{lem}\label{lem:symmetry} For any $s\in \mathbb{Z}$,
\begin{equation}\label{eqn:calHd}
\mathcal{H}^{ds}\circ m_d =\lambda \cdot\mathcal{H}^s.
\end{equation}
Moreover, for any $\mathbf{i}\in G_0$, we have
\begin{equation}\label{eqn:grp2}
\mathcal{H}^s(\mathbf{i}+\mathbf{k},x)=\mathcal{H}^s(\mathbf{k},x), \quad \forall\,(\mathbf{k},x)\in\Sigma_d\times[0,1).
\end{equation}
\end{lem}
\begin{proof} We first prove (\ref{eqn:calHd}). Note that for each $y\in \mathbb{R}$, $m\ge 1$, $k_0\in \{0,1,\ldots, d-1\}$ and $\mathbf{k}\in \Sigma_d$, we have
\[
\xi_m(k_0\mathbf{k},dy-k_0)= \xi_{m-1}(\mathbf{k},y),
\]
where $\xi_0(\mathbf{k},y)=y$. So for $k_0=\lfloor dx\rfloor$,
%\marginpar{$F(k_0\mathbf{k},\xi_m(d(x+s)-k_0))$ is corrected to $F(\xi_m(k_0\mathbf{k},d(x+s)-k_0))$.}
\begin{align*}
\mathcal{H}^{ds} \circ m_d(\mathbf{k},x) & =\mathcal{H}^{ds} ( k_0\mathbf{k}, dx-k_0)\\
& =H_{k_0\mathbf{k}}(d(x+s)-k_0)-H_{k_0\mathbf{k}}(dx-k_0)\\
& =\sum_{m=1}^\infty \lambda^m \cdot \left[F(\xi_m(k_0\mathbf{k},d(x+s)-k_0))-F(\xi_m( k_0\mathbf{k}, dx-k_0))\right]\\
& =\sum_{m=1}^\infty \lambda^m \cdot \left[F(\xi_{m-1} (\mathbf{k},x+s))-F(\xi_{m-1}(\mathbf{k},x))\right]\\
&= \lambda \cdot (F(x+s)-F(x))+\lambda \cdot \mathcal{H}^s (\mathbf{k},x)\\
&= \lambda \cdot \mathcal{H}^s (\mathbf{k},x),
\end{align*}
where the last equality follows from the fact that $F$ is $1$-periodic and $s\in \mathbb{Z}$.

Now let us prove (\ref{eqn:grp2}). Given $\mathbf{i}\in G_0$, by Lemma~\ref{lem:add}, $h_{\mathbf{i}+\mathbf{k}}\equiv h_{\mathbf{k}}$, and hence
$H_{\mathbf{i}+\mathbf{k}}\equiv H_{\mathbf{k}}$ for all $\mathbf{k}\in \Sigma_d$. The equality follows.
\end{proof}
%the functions $\mathcal{H}^s$ satisfies the following identity:

%
%$$X_d=\{\{x_m\}_{m=0}^\infty: x_m\in \mathbb{R}/\mathbb{Z},\, \mathbf{m}_d (x_{m+1})=x_m,\forall m\ge 0\}$$ of backward orbits of $\mathbf{m}_d: %x\mapsto dx\mod 1$
%as follows:
%$$\xi(\mathbf{k},x)=(\xi_n(\mathbf{k},x))_{n=0}^\infty, \mbox{ where }\xi_n(\mathbf{k},x)= \frac{x+k_1+k_2d+\cdots+ k_n d^{n-1}}{d^n} \mod 1.$$
%As a closed subgroup of the compact abelian group $(\mathbb{R}/\mathbb{Z})^{\mathbb{Z}_{\ge 0}}$, the space $X_d$ is itself a compact abelian group. %Let $\mu$ denote the Haar measure on $X_d$. Then for any open set $U$ of $\mathbb{R}/\mathbb{Z}$ and any integer $m\ge 0$,
%$$\mu\left(\{(\xi_n)_{n=0}^\infty\in X_d: \xi_m\in U\} \right)=|U|,$$
%where $|U|$ is the standard length of $U$ in $\mathbb{R}/\mathbb{Z}$. Note that $\mu$ is invariant under the map
%$$m_d:X_d\to X_d, \, m_d((\xi_m)_{m\ge 0})= (d \xi_m\mod 1)_{m\ge 0}.$$

%Note also that for each $m\ge 0$,
%\begin{equation}\label{eqn:chm}
%\xi_{m+1}\circ m_d=\xi_m.
%\end{equation}

%The functions $\mathbf{H}^s$ satisfies the following identity:
%$$H^{ds} \circ m_d = \lambda (f(\xi_0+s)-f(\xi_0))+\lambda\lambda H^s.$$
%Whenever $m\ge 1$, since $E_{m-1,q} \circ m_d=E_{m,q}$, we have
%\begin{equation}\label{eqn:mrec}
%c_{m-1,q}^{ds}=\lambda c_{m,q}^s.
%\end{equation}

We shall also need the following two lemmas. %\marginpar{$G_0$ is corrected to $\Sigma_d$.}
\begin{lem} \label{lem:bwirr}
For $\mathbf{i}=(i_n)_{n\ge 1}\in \Sigma_d\setminus\{\mathbf{0}\}$, let $z_m=(i_1+i_2d +\cdots +i_m d^{m-1})/d^m$ for each $m\ge 1$. Then for any integer $q\not=0$, there exists a positive integer $m_*$ such that for any $m\ge m_*$,
$$qz_m\not=0\mod 1.$$
\end{lem}
\begin{proof} Fix $q$ and let $Y=\{\frac{k}{q}\mod 1: k\in \mathbb{Z}\}$. Then $Y$ is a $\mathbf{m}_d$-invariant finite set. Assuming that the conclusion fails, i.e., $z_m\in Y$ for all $m\ge 1$,  it remains to show that $\mathbf{i}=\mathbf{0}$. Since $\mathbf{m}_d(z_{m+1})=z_{m}$ for all $m\ge 1$, it follows that each $z_m$ is a periodic point of $\mathbf{m}_d$ in $Y$. Since each $z_m$ is eventually mapped to $0\mod 1$, it follows that $z_m=0\mod 1$ for all $m\ge 1$, and hence $\mathbf{i}=\mathbf{0}$.
\end{proof}

%We also need the lemma below.

\begin{lem}\label{lem:const}
  If $\mathcal{H}^1$ is constant, then $h_{\mathbf{0}}$ (or equivalently $H_{\mathbf{0}}$) is $1$-periodic.
\end{lem}

\begin{proof}
    Let $c\in\mathbb{C}$ be such that  $\mathcal{H}^1\equiv c$. By definition,  $H_{\mathbf{k}}(x+1)-H_{\mathbf{k}}(x)=c$ for any $(\mathbf{k},x)\in\Sigma_d\times[0,1)$. Since $H_{\mathbf{k}+\iota(m)}(y)=H_{\mathbf{k}}(y+m)$ for any $y\in\mathbb{R}$, any integer $m\ge 0$ and any $\mathbf{k}\in\Sigma_d$, it follows that $H_{\mathbf{k}}(y+1)-H_{\mathbf{k}}(y)=c$ for any $y\in\mathbb{R}$ and any $\mathbf{k}\in \Sigma_d$. In particular, $\mathcal{H}^s\equiv sc$ for any $s\in\mathbb{Z}$. By (\ref{eqn:calHd}), it follows that
    %On the other hand, note that $H_{\mathbf{0}}(dx) = \lambda(F(x)-F(0) + H_{\mathbf{0}}(x))$ for any $x\in\mathbb{R}$ and $F$ is $1$-periodic. It %follows that
   \[
   dc= \lambda c.
   \]
By the Mean Value Theorem,
$\lambda_0\le d$, so $|\lambda|<\lambda_0^\alpha\le d$. Thus $c=0$, which completes the proof.
\end{proof}

We are ready to complete the proof of Proposition~\ref{prop:C1(i)(iii)}.
\begin{proof}[Completion of proof of Proposition~\ref{prop:C1(i)(iii)}]
Assume that (i) holds. By Lemma~\ref{lem:const}, it suffices to show that $\mathcal{H}^s$ is constant for any $s\in\mathbb{Z}$. To this end, let
$$E_{m,q}(\mathbf{k},x)=e^{2\pi i \xi_m(\mathbf{k},x) q}, \quad  m, q\in\mathbb{Z}, m\ge 0.$$
We shall show that for any integers $m\ge 0$, $q\not=0$ and $s$,
\begin{equation}\label{eqn:cmq}
c_{m,q}^s:=\int_{\Sigma_d\times[0,1) } \mathcal{H}^s \cdot \overline{E_{m,q}} \dif \mu=0.
\end{equation}

Before the proof of (\ref{eqn:cmq}), let us show how it implies that $\mathcal{H}^s$ is constant. Note that $\mathcal{H}^s$ is bounded and continuous on $\Sigma_d\times[0,1)$. Let $\mathcal{B}$ denote the Borel $\sigma$-algebra of $[0,1)$ and let $\mathcal{B}_m=\xi_m^{-1}(\mathcal{B})$, which is a $\sigma$-algebra in $\Sigma_d\times[0,1)$. Then $\mathcal{B}_m$ is monotone increasing to the Borel $\sigma$-algebra in $\Sigma_d\times[0,1)$.
%\marginpar{$[0,1)\times\Sigma_d$ is corrected to $\Sigma_d\times[0,1)$.}
By the Martingale Convergence theorem, $\mathcal{H}^s_m:=\mathbb{E}[\mathcal{H}^s|\mathcal{B}_m]$ converges $\mu$-a.e. to $\mathcal{H}^s$.  Since $E_{m,q}$ is $\mathcal{B}_m$-measurable,
$$\int_{\Sigma_d\times [0,1)} \mathcal{H}^s_m \cdot \overline{E_{m,q}} \dif \mu= c_{m,q}^s=0$$
for any $m\ge 0$ and $q\not=0$, where the last equality follows from (\ref{eqn:cmq}).
Since $(\xi_m)_* \mu$ is the standard Lebesgue measure on $[0,1)$, each $\mathcal{H}^s_m$, which can be viewed as a function in $[0,1)$, must be constant a.e.. Thus $\mathcal{H}^s$ is constant $\mu$-a.e.. Since $\mathcal{H}^s$ is continuous, it is constant.

%\[
%c_{m,q}^s=\int_{\Sigma_d\times[0,1) } \mathcal{H}^s(\mathbf{k},x) \overline{E_{m,q}(\mathbf{k},x)} \dif \mu.
%\]
It remains to prove that (\ref{eqn:cmq}) holds for all $m\ge 0$ and $q\not=0$. First let us show
\begin{equation}\label{eqn:cmrec}
c_{m+1,q}^{ds}=\lambda c_{m,q}^{s}.
\end{equation}
Indeed, since $(m_d)_*\mu=\mu$ and $\xi_{m+1}\circ m_d=\xi_m$, we have
\begin{multline*}
\int \mathcal{H}^{ds} \circ m_d \cdot\overline{E_{m,q}} \,\dif \mu
= \int \mathcal{H}^{ds} \circ m_d \cdot \overline{E_{m+1,q}\circ m_d}\, \dif \mu
=\int \mathcal{H}^{ds} \cdot \overline{E_{m+1,q}} \,\dif\mu= c_{m+1,q}^{ds}.
\end{multline*}
Combining this with \eqref{eqn:calHd}, we obtain \eqref{eqn:cmrec}.

Next, by Lemma~\ref{lem:add}, $G_0\not=\{\mathbf{0}\}$. Let $\mathbf{i}=(i_n)_{n\ge 1}\in G_0\setminus \{\mathbf{0}\}$. Then by \eqref{eqn:grp2}, for any $s\in \mathbb{Z}$, $m\ge 1$ and $q\in \mathbb{Z}$,
$$c_{m,q}^s=e^{2\pi i z_m q} c_{m,q}^s,$$
where $z_m=(i_1+i_2 d+ \cdots +i_m d^{m-1})/d^m$.
By Lemma~\ref{lem:bwirr}, for each $q\not=0$, there exists $m_*(q)$ such that for any $m\ge m_*(q)$, $qz_m\not=0 \mod 1$, so that $e^{2\pi i z_m q}\not=1$, which implies that $c_{m,q}^s=0$ for all $s\in \mathbb{Z}$. By (\ref{eqn:cmrec}), it follows that $c_{m,q}^s=0$ for all $m\ge 0$ and $0\ne q\in\mathbb{Z}$. The proof is completed.
\end{proof}
%\equiv 0$ and hence $h_{\mathbf{i}}\equiv 0$ for all $\mathbf{i}$.
%$$H^s(\mathbf{k},x)=\sum_{m=1}^\infty \left(f\left(\frac{s+x+k_1+k_2 d+ \cdots+ k_n d^{n-1}}{d^n}\right)-f\left(\frac{x+k_1+k_2 d+ \cdots+ k_n %d^{n-1}}{d^n}\right)\right)=h_{\mathbf{k}}(x+s)-h_{\mathbf{k}}(x).$$

\section{Proof of the Main Theorem}\label{sec:pfmainthm}
This section is devoted to the proof of the Main Theorem. The following is an equivalent reformulation of the Main Theorem.

\begin{mainthm'} Let $T:\mathbb{R}/\mathbb{Z}\to \mathbb{R}/\mathbb{Z}$  be a real analytic expanding map and let $f:\mathbb{R}/\mathbb{Z}\to \mathbb{R}$ be a real analytic function. Assume that $f$ is not real analytically cohomologous to constant. Let $\mu$ be a maximizing measure of $(T,f)$. Then the measure-theoretic entropy $h_T(\mu)=0$.
\end{mainthm'}

In \S\S~\ref{subsec:idea}--~\ref{subsec:compS}, we shall prove the Main Theorem' under the following technical condition:
$$(\ast): T \mbox{ is orientation-perserving and } \text{Fix}(T)\setminus \text{supp}(\mu) \not=\emptyset,$$
where $\text{Fix}(T)$ denotes the set of fixed points of $T$. In \S~\ref{subsec:iteration}, we shall show how to  remove this condition and complete the proof, using the equivalence of the maximization problem between $(T,f)$ and $(T^k, f+f\circ T+\cdots+f\circ T^{k-1})$.

\subsection{Strategy of the proof assuming \texorpdfstring{($\ast$)}{(*)}}\label{subsec:idea}
Let $d$ be the degree of $T$. Without loss of generality, we may assume that $0\in\mathbb{R}/\mathbb{Z}$ is a fixed point of $T$ which is not contained in $\text{supp}(\mu)$. As before, let $\whT$ be the unique lift of $T$ with $\whT(0)=0$ and denote $\tau:=\whT^{-1}$. Moreover, denote $\tau_i(x):=\tau(x+i)$ for $0\le i<d$. Let $S=\text{supp}(\mu)$.
Identifying $\mathbb{R}/\mathbb{Z}$ with $[0,1)$ in the natural way, $S$ is a non-empty compact subset of $(0,1)$  with $T(S)=S$.
%Then $\{\tau_i:0\le i<d\}$ is a contracting IFS with attractor $[0,1]$, and $\tau_i([0,1])=[x_i,x_{i+1}]$ for $0\le i<d$, where $x_i:=\tau(i)$, %$0\le i\le d$.
%\subsection{Preliminaries}
%Let $T$ be orientation preserving with degree $d\ge 2$. Up to a linear change of coordinate, we may assume that $T(0)=0$.

For $\mathbf{i}=(i_n)_{n\ge 1}\in \Sigma_{d}$, recall
\[
\tau_{\mathbf{i},n}=\tau_{i_n}\circ\cdots\circ\tau_{i_2}\circ\tau_{i_1}, \quad \forall n\ge 1.
\]
In order to show that $h_T(\mu)=0$, we shall analyze the inverse limit of $T:S\to S$. So let
$$\mathcal{S}=\{(\mathbf{i}, x)\in \Sigma_d \times S: \tau_{\mathbf{i},n}(x)\in S, \forall n\ge 1\}.$$
Given $x\in S$, denote
\[
\mcls^x:=\{\mathbf{i}\in\Sigma_d: (\mathbf{i},x)\in\mcls\}.
\]
Note that $\mcls^x$ is a non-empty compact subset in  $\Sigma_d$. Given $\mathbf{i}\in \Sigma_d$, denote
\[
\mcls_{\mathbf{i}}:=\{x\in S:(\mathbf{i},x)\in\mcls\}=\{x\in S: \mathbf{i}\in \mcls^x\}.
\]
Then $\mcls_{\mathbf{i}}$ is a compact subset of $S$ (possibly empty).

Let us apply
Theorem~\ref{thm:transversal}, so that
$$h_{\mathbf{i}}(x)=\sum_{n=1}^\infty (\widehat{f}\circ \tau_{\mathbf{i},n}(x)-\widehat{f}\circ \tau_{\mathbf{i},n}(0))), \quad x\in\mathbb{R}.$$
As we are assuming that $f$ is not real analytically cohomologous to constant, the second alternative of Theorem~\ref{thm:transversal} holds. So for distinct $\mathbf{i}, \mathbf{j}\in \Sigma_d$, the real analytic functions $h_{\mathbf{i}}$ and $h_{\mathbf{j}}$ are not identical, and hence $h'_{\mathbf{i}}-h'_{\mathbf{j}}$ has only isolated zeros. This allows us to define, for each $x\in S$, two total orders $\prec^x_+$ and $\prec^x_-$ on $\mcls^x$ as follows.

\begin{defn}
Given $x\in S$, define two total orders $\prec^x_\pm$ on $\mcls^x$ as follows. Given $\mathbf{i}\ne \mathbf{j}$ in $\mcls^x$,
\begin{itemize}
  \item $\mathbf{i}$ is {\em strictly less than} $\mathbf{j}$ with respect to $\prec^x_+$ if there exists $\delta>0$ such that $h'_{\mathbf{i}}(y)<h'_{\mathbf{j}}(y)$ holds for all $y\in (x, x+\delta)$;
  \item $\mathbf{i}$ is {\em strictly less than} $\mathbf{j}$  with respect to $\prec^x_-$ if there exists $\delta>0$ such that $h'_{\mathbf{i}}(y)>h'_{\mathbf{j}}(y)$ holds for all $y\in (x-\delta, x)$.
      %here exists a minimal  $m\ge 0$ such that $(-1)^mh_\imath^{(m)}(x)<(-1)^m h_\jmath^{(m)}(x)$.
\end{itemize}
\end{defn}

\begin{lem}\label{lem:intersection}
 Suppose that $f$ is not real analytically cohomologous to constant. Then for any $x\in S$, $(\mcls^x, \prec^x_+)$ has a unique maximal element $\kappa_+(x)\in \mcls^x$, and $(\mcls^x, \prec^x_-)$ has a unique maximal element $\kappa_-(x)\in \mcls^x$, which define two maps $\kappa_\pm: S\to\Sigma_d$.
\end{lem}

\begin{proof}
  Uniqueness follows from the fact that both orders are total orders. For existence, let us focus on ``$\prec^x_+$"; the discussion on ``$\prec^x_-$" is totally similar and omitted.  For each $m\ge 1$, since $h_{\mathbf{i}}^{(m)}(x)$ is continuous in $\mathbf{i}\in \mathcal{S}^x$, $\mcli_m$ defined inductively below is a decreasing sequence of non-empty compact subsets of $\mathcal{S}^x$, where $\mcli_0=\mathcal{S}^x$:
  \[
  \mcli_m=\left\{\mathbf{i}\in \mcli_{m-1} : h_{\mathbf{i}}^{(m)}(x) = \max\{h_{\mathbf{j}}^{(m)}(x):\mathbf{j}\in \mcli_{m-1}\}\right\}.
  \]
Therefore $\bigcap_{m=1}^\infty \mcli_m$ contains at least one element $\mathbf{i}$. Then for any $\mathbf{j}\in \mathcal{S}^x$, $h_{\mathbf{i}}^{(m)}(x)\ge h_{\mathbf{j}}^{(m)}(x)$ for any $m\ge 1$, which implies that $\mathbf{i}$ is a maximal element in $\mathcal{S}^x$ with respect to $\prec^x_+$.
\end{proof}

%\begin{proof}
%  Existence easily follows from compactness of $\mathcal{S}^x$ and continuity of the functions $\mathbf{i}\mapsto h_{\mathbf{i}}(y)$, $y\in \mathbb{R}$. Uniqueness follows from the fact that the orders are total orders.
%\end{proof}

We shall prove the following proposition in \S~\ref{subsec:compS}.
\begin{prop} \label{prop:limitpt}
Assume that $\mathcal{S}_{\mathbf{i}}$ has a limit point $x$. Then $\mathbf{i}\in \{\kappa_+(x), \kappa_-(x)\}.$
\end{prop}

\subsection{Proof of the Main Theorem' assuming \texorpdfstring{($\ast$)}{(*)}}\label{subsec:proofmainstar}
Let $\mathcal{T}:\mathcal{S}\to \mathcal{S}$ be the inverse limit of $T:S\to S$, i.e., for $x\in \tau_{i_0}([0,1))\cap S$,
$$\mathcal{T}((i_n)_{n\ge 1},x)=((i_{n-1})_{n\ge 1}, \whT(x)-i_0). $$
%\text{\color{red}(does $\mathbf{i}i_0$ need explanation?)}$$
As $0\not\in S$, $T:S \to S$ is naturally topologically conjugate to a one-sided subshift with $d$ symbols
%(\textcolor{red}{``conjugate to" or ``embedded into"? To me the term ``horseshoe" is usually understood as sub-shift of finite type up to %isomorphism. I suggest here replacing the word ``horseshoe" with ``(sub-)shift".}) a one-sided horseshoe
and $\mathcal{T}:\mathcal{S}\to\mathcal{S}$ is topologically conjugate to a two-sided subshift with $d$ symbols.

We shall need the following well-known result.
\begin{prop}\label{prop:uncountable} If $\nu$ is an ergodic invariant probability Borel measure of the two-sided full shift $\sigma: \{0,1,\ldots,d-1\}^\mathbb{Z}\circlearrowleft$ such that $h_\sigma(\nu)>0$, then for any Borel subset $U$ of $\{0,1,\ldots, d-1\}^{\mathbb{Z}}$ with $\nu(U)>0$, the following holds for $\nu$-a.e. $(i_n)_{n\in\mathbb{Z}} \in U$:
$$\{(j_n)_{n\in\mathbb{Z}} \in U: j_n=i_n, \,\forall n\ge 0\},$$
$$\{(j_n)_{n\in\mathbb{Z}} \in U: j_n=i_n,\, \forall n<0\}$$
are both uncountable.
\end{prop}
\begin{proof} It is well-known that the two-sided full shift is topologically conjugate to a linear horseshoe $F:\Lambda\to\Lambda$, with the sets in question corresponding to (local) stable and unstable manifolds. If $\nu_F$ is the $F$-invariant ergodic probability measure corresponding to $\nu$, then $h_F(\nu_F)>0$. It is well-known that (see e.g. \cite{LY85}) for $\nu_F$-a.e. $(y,x)\in \Lambda$, the conditional measures $\nu_F^x$ and $\nu_{F,y}$ along the local stable and unstable manifolds have positive local dimensions and thus admit no atom. So if $\nu_F(W)>0$, then for $\nu_F$-a.e. $z\in W$, the intersection of $W$ with the stable and unstable manifolds of $z$ must be both uncountable.
\end{proof}
The following lemma deals with the measurability issue involved.
\begin{lem}\label{lem:Borel}
  The set $B$ defined below is Borel:
  \[
  B:=\{(\mathbf{i},x)\in \mcls : x~\text{is a limit point of } \mcls_\mathbf{i}\}.
  \]
\end{lem}
%\textcolor{red}{The proof below has been rewritten.}
\begin{proof}
 Define $\phi:\Sigma_d\times S\to \mathbb{R}$ as $\phi(\mathbf{i},x)=\dist(x,(\mcls_{\mathbf{i}}\cup\{2\})\setminus\{x\})$ (here use $\mcls_{\mathbf{i}}\cup\{2\}$ instead of $\mcls_{\mathbf{i}}$ because $\mcls_{\mathbf{i}}\setminus\{x\}$ might be empty), where $\dist(x,\cdot)$ means the distance of $x$ to a subset of the real line in the usual sense. Then $B=\phi^{-1}(0)$. It suffices to show that $\phi$ is Borel. In fact, the following hold.
  \begin{itemize}
    \item Given $\mathbf{i}\in \Sigma_d$, $x\mapsto \phi(\mathbf{i},x)$ is continuous. This is easy to check.
    \item Given $x\in S$, $\mathbf{i}\mapsto \phi(\mathbf{i},x)$ is Borel. To see this, for each $n\ge 1$, let
    \[
    \varphi_n(\mathbf{i}):=\dist\left(x,(\mcls_{\mathbf{i}}\cup\{2\})\setminus (x-\tfrac{1}{n},x+\tfrac{1}{n})\right), \quad \mathbf{i}\in \Sigma_d.
    \]
    Then, by compactness of $\mathcal{S}$, $\varphi_n$ is lower semi-continuous for each $n$, while $\lim_{n\to\infty}\varphi_n(\mathbf{i})=\phi(\mathbf{i},x)$.
  \end{itemize}
  It follows that $\phi$ is Borel.
\end{proof}

Let us now complete the proof of the Main Theorem assuming ($\ast$).

%(\textcolor{red}{The second paragraph of proof below has been rewritten.})
\begin{proof}[Proof of the Main Theorem' assuming ($\ast$)] Assume by contradiction that $h_T(\mu)>0$. By ergodic decomposition and affinity of the entropy function, we may assume that $\mu$ is ergodic with respect to $T$.  As the inverse limit of $T:S\to S$, the map $\mathcal{T}:\mathcal{S}\to \mathcal{S}$ has an ergodic invariant measure $\widetilde{\mu}$ with positive entropy.

We shall show that $\widetilde{\mu}(\mathcal{S})=0$ to get a contradiction. To this end, let $B$ be the set defined in Lemma~\ref{lem:Borel}. By the definition of $B$ and Proposition~\ref{prop:limitpt}, $\mathbf{i}\in \{\kappa_{\pm}(x)\}$ for any $(\mathbf{i},x)\in B$, so by Proposition~\ref{prop:uncountable}, $\widetilde{\mu}(B)=0$. By the definition of $B$ again, for any $(\mathbf{i},x)\in \mcls\setminus B$, the set $\{y\in S: (\mathbf{i},y)\in \mcls\setminus B\}$ is countable,
%\marginpar{``$\mcls_{\mathbf{i}}$ is countable for any $(\mathbf{i},x)\in \mcls\setminus B$" is corrected to the red-colored statement.}
so by Proposition~\ref{prop:uncountable} again, $\widetilde{\mu}(\mcls\setminus B)=0$. The proof is done.
\end{proof}

\subsection{Complexity of \texorpdfstring{$\mathcal{S}$}{the inverse limit}}\label{subsec:compS}
This subsection is devoted to the proof of Proposition~\ref{prop:limitpt}.
Let $g:\mathbb{R}/\mathbb{Z}\to\mathbb{R}$ be a {\em sub-action} for $(T,f)$, i.e., $g:\mathbb{R}/\mathbb{Z}\to\mathbb{R}$ is a continuous function such that
$$f(x)\le g\circ T(x)- g(x)+\beta(f)$$
for all $x\in \mathbb{R}/\mathbb{Z}$. If we put
\begin{equation}\label{eq:level set}
  S_0=\{x\in\mathbb{R}/\mathbb{Z}: f(x)= g(T(x))-g(x)+\beta(f)\},
\end{equation}
then a $T$-invariant measure is a maximizing measure $f$ if and only if it is supported in $S_0$.
In particular, $S\subset S_0$.

Sub-actions played an important role in the ergodic maximization problem. When $T$ is expanding and $f$ is Lipschitz, it is well-known that there exists a sub-action $g$ which is Lipschitz.  This is often referred to as {\em M\~an\'e's lemma.} However, in general we cannot expect higher regularity of $g$, even when we assume $f$ and $T$ are both real analytic, see \cite{BJ02}.
%In \cite{Gao21}, it has been shown in the case $T(x)=dx \mod 1$ with $d\ge 2$ and $f$ is a sum of convex function and a quadratic function %\textcolor{red}{($f$ is considered as a function on $\mathbb{R}$??? Such a function is called {\em semi-convex} in some literatures. \cite{Gao21} has %been cited before, so is it necessary to mention it again here?)}, one can find $g$ in the same form.
\begin{lem}\label{lem:holonomy}
  Let $(\mathbf{i},x)\in\mathcal{S}$. Then for any $y\in S$,
  %,y\in [0,1]$. Let $\imath\in \mcls^x$ and $\jmath\in \mcls^y$. Then
  \[
  g(y)-g(x)\ge h_{\mathbf{i}}(y)-h_{\mathbf{i}}(x).
  %\sum_{n=1}^\infty \big(f(\tau_{\mathbf{i},n}y)-f(\tau_{\mathbf{i},n}x)\big).% \sum_{n=1}^\infty %\big(f(\tau_{\imath,n}y)-f(\tau_{\imath,n}x)\big) \le g(y)-g(x)\le
  \]
  Moreover, equality holds if  $(\mathbf{i},y)\in \mathcal{S}$.
\end{lem}

%\begin{proof}
\begin{proof}
Denote $\tau_{\mathbf{i}, 0}=\mathrm{id}$. Since $g$ is a sub-action and $x\in\mathcal{S}_{\mathbf{i}}$, for each $n\ge 0$,
$$g(\tau_{\mathbf{i}, n}(x))= f(\tau_{\mathbf{i}, n+1}(x))+ g(\tau_{\mathbf{i}, n+1}(x))-\beta(f),$$
$$g(\tau_{\mathbf{i}, n}(y))\ge f(\tau_{\mathbf{i}, n+1}(y))+ g(\tau_{\mathbf{i}, n+1}(y))-\beta(f).$$
Moreover, equality holds in the last inequality if $y\in\mathcal{S}_{\mathbf{i}}$.
Therefore, for each $m\ge 1$,
$$g(x)=\sum_{n=1}^m f(\tau_{\mathbf{i}, n}(x))+ g(\tau_{\mathbf{i}, m}(x))-m\beta(f),$$
$$g(y)\ge \sum_{n=1}^m f(\tau_{\mathbf{i}, n}(y))+ g(\tau_{\mathbf{i}, m}(y))-m\beta(f).$$
Consequently,
\[
g(y)-g(x)\ge \sum_{n=1}^m \big(f(\tau_{\mathbf{i},n}(y))-f(\tau_{\mathbf{i},n}(x))\big)  + \big(g(\tau_{\mathbf{i},m}(y))-g(\tau_{\mathbf{i},m}(x))\big).
\]
Letting $m\to\infty$, we obtain the desired inequality.

If $y\in \mathcal{S}_{\mathbf{i}}$, then all the inequalities above become equalities.
\end{proof}

\begin{proof}[Proof of Proposition~\ref{prop:limitpt}]
  We may assume that there exists a sequence $(x_k)$  in $\mathcal{S}_{\mathbf{i}}$  converging to $x$ from the right side; the other situation is similar and omitted. We shall show that $\mathbf{i}=\kappa_+(x)$. To this end, consider an arbitrary $\mathbf{j}\in\mathcal{S}^x$.
  %$\mathbf{j}\in \mathcal{S}^x$.
  By Lemma~\ref{lem:holonomy}, we have the following.
\begin{itemize}
  \item Since $x,x_k\in \mathcal{S}_{\mathbf{i}}$,
  $$g(x_k)-g(x)=h_{\mathbf{i}} (x_k)-h_{\mathbf{i}}(x)=\int_x^{x_k} h'_{\mathbf{i}}(t) \dif t.$$
  \item Since $x\in \mathcal{S}_{\mathbf{j}}$,
  $$g(x_k)-g(x) \ge h_{\mathbf{j}} (x_k)-h_{\mathbf{j}} (x)=\int_x^{x_k} h'_{\mathbf{j}}(t) \dif t.$$
\end{itemize}
Therefore,
  $$\int_x^{x_k} h'_{\mathbf{j}}(t) \dif t \le \int_x^{x_k} h'_{\mathbf{i}}(t) \dif t.$$
 Since $x_k$ converges to $x$ from the right side, the inequality above implies that $h'_{\mathbf{j}}(t)\le h'_{\mathbf{i}}(t)$ holds for a sequence of points converging to $x$ from the right side.
 It follows that $\mathbf{i} \prec_+^x \mathbf{j}$ cannot hold. Therefore $\mathbf{i}=\kappa_+(x)$.
\end{proof}

\subsection{Iteration}\label{subsec:iteration}
We shall show that the technical condition ($\ast$) can be removed.
Let $T\in\mcle^\omega$ and let $f\in C^\omega(\mathbb{R}/\mathbb{Z})$. For any positive integer $k$, let $\mathscr{S}_k f= f+f\circ T+\cdots+ f\circ T^{k-1}$. %The maximizing problems of $(T,f)$ and $(\mathscr{S}_k f, T^k)$ are equivalent.

The following elementary observations should be well-known. Let us include a proof for completeness.

\begin{lem}\label{lem:iterates}
\mbox{}
\begin{enumerate}
\item There exists $g\in C^\omega(\mathbb{R}/\mathbb{Z})$ such that $\mathscr{S}_k f= g\circ T^k-g +const$ for some $k\ge 1$ if and only if the same holds for $k=1$.
\item A maximizing measure of $(T,f)$ is also a maximizing measure of $(T^k, \mathscr{S}_k f)$ for any $k\ge 1$.
\end{enumerate}
\end{lem}

\begin{proof}
(1) If $f=g\circ T- g+ c$, then $\mathscr{S}_k f= g \circ T^k- g+ kc.$ For the other direction,  suppose that $\mscs_k f=g\circ T^k-g+c$ for some $g\in C^\omega(\mathbb{R}/\mathbb{Z})$. Then
  \[
  f\circ T^k -f =(\mscs_kf)\circ T - \mscs_kf =(g\circ T^k-g)\circ T - (g\circ T^k-g),
  \]
  which can be rewritten as follows:
  \[
  (f+g-g\circ T)\circ T^k =f+g-g\circ T.
  \]
  Since $T^k:\mbbt\to\mbbt$ is topologically transitive and since $f+g-g\circ T$ is continuous, the equality above implies that $f+g-g\circ T=const$.

(2) If $\nu$ is a maximizing measure for $(T^k, \mathscr{S}_k f)$, then $\mu= \frac{1}{k} \sum_{j=0}^{k-1} T^j_* \nu$ is $T$-invariant, and
$$\int f \dif\mu= \frac{1}{k} \sum_{j=0}^{k-1} \int f\circ T^j \dif\nu=\int \frac{1}{k}\mathscr{S}_k f \dif\nu.$$
This shows $\beta(T,f)\ge \beta(T^k, \frac{1}{k}\mathscr{S}_k f)$. On the other hand, if $\mu$ is a maximizing measure of $(T,f)$, then $\mu$ is $T^k$-invariant and
 $$\beta(T,f)=\int f \dif\mu= \int \frac{1}{k} \mathscr{S}_k f \dif\mu\le \beta(T^k, \frac{1}{k}\mathscr{S}_k f).$$
Hence $\beta(T,f)= \beta(T^k, \frac{1}{k}\mathscr{S}_k f)$ and a maximizing measure for $(T,f)$ is also maximizing for $(T^k, \mathscr{S}_k f)$.
%%%the paragraph above is a revision of the paragraph below.

%For each $T$-invariant Borel probability measure $\mu$, $\mu$ is $T^k$-invariant and
% $$\int f \dif\mu= \int \frac{1}{k} \mathscr{S}_k f \dif\mu.$$
%So $\beta (T,f)\le \beta(\frac{1}{k}\mathscr{S}_k f, T^k)$. On the other hand, if $\nu$ is a maximizing measure for $(T^k, \mathscr{S}_k f)$, then $\mu= \frac{1}{k} \sum_{j=0}^{k-1} T^j_* \nu$ is $T$-invariant, and
%$$\int f \dif\mu= \frac{1}{k} \sum_{j=0}^{k-1} \int f\circ T^j \dif\nu=\int \frac{1}{k}\mathscr{S}_k f \dif\nu.$$
%This shows $\beta(T,f)\ge \beta(\frac{1}{k}\mathscr{S}_k f, T^k)$. Hence $\beta(T,f)= \beta(\frac{1}{k}\mathscr{S}_k f, T^k)$ and a $(T,f)$-maximizing measure is also maximizing for $(T^k, \mathscr{S}_k f)$.
\end{proof}

\begin{proof}[Completion of the proof of the Main Theorem']  As $f$ is not analytically cohomologous to a constant, for $S_0$ defined by \eqref{eq:level set}, $S_0\subsetneq \mathbb{R}/\mathbb{Z}$, so $S=\text{supp}(\mu)$ is nowhere dense in $\mathbb{R}/\mathbb{Z}.$ In particular, there is a periodic point $p$ of $T$ such that $p\not\in S$. Let $k$ be an even positive integer such that $T^k(p)=p$. By Lemma~\ref{lem:iterates}, $\mu$ is a maximizing measure for $(T^k, \mathscr{S}_k f)$ and $\mathscr{S}_k f$ is not analytically cohomologous to a constant with respect to $T^k$. Thus by what we have proved before, $h_{T^k}(\mu)=0$ and hence $h_T(\mu)=0$.
\end{proof}

\section{Proof of Theorems~\ref{thm:Cr}, \ref{thm:analytic Lyp} and ~\ref{thm:Cr Lyp}}
\label{sec:coro}
%(\textcolor{red}{This section is almost unchanged. I think maybe we should not provide a detailed proof of the proposition below because similar %things had almost been done in \cite{Bre08,JM08}.})
In this section, we prove Theorems~\ref{thm:Cr},~\ref{thm:analytic Lyp} and ~\ref{thm:Cr Lyp}.
The basic idea is to approximate $C^r$ functions or maps with real analytic ones and then apply the Main Theorem.

We shall need upper semi-continuity of the following function:
\begin{equation}\label{eq:entropy func}
\mclh(T,f):=\sup_{\mu\in \mclm_{max}(T,f)} h_T(\mu), \,\, (T, f)\in \mcle^1\times C^0(\mathbb{R}/\mathbb{Z}).
\end{equation}
Here $\mclm_{max}(T,f)$ denote the collection of $(T,f)$ maximizing measures.  This result is essentially contained in \cite{Bre08, JM08}.

%We shall need the following proposition which is more or less well-known. For fixed $T$, the upper semi-continuity of $f\mapsto \mclh(T,f)$ was %shown in \cite{Bre08}; see the first paragraph in the section of ``Proof of Theorem~1.2". For upper semi-continuity of $(T,f)\mapsto\mclh(T,f)$ or %$T\mapsto\mclh(T,\log|T'|)$ , it was implicitly shown in \cite{JM08}.

\begin{prop}\label{prop:usc}
The function $\mclh$ is upper semi-continuous on $\mcle^1\times C^0(\mathbb{R}/\mathbb{Z})$.
%
%The same conclusion holds for $\mclm_{max}$ instead of $\mclm_{min}$.
\end{prop}

%\begin{rmk}
%  The proposition above should be known before.
%\end{rmk}

\begin{proof}
%  We only need to consider the $\mclm_{max}$ case, because $\mclm_{min}(f)=\mclm_{max}(-f)$.
It was observed in \cite{Bre08} that for any fixed $T\in\mcle^1$, $f\mapsto \mclh(T,f)$ is upper semi-continuous on $C^0(\mathbb{R}/\mathbb{Z})$. Indeed, if $\mu_n$ is a maximizing measure for $(T, f_n)$ with $f_n\to f$ in $C^0(\mathbb{R}/\mathbb{Z})$, then any accumulation point $\mu$ of $\mu_n$ in the weak-* topology is a maximizing measure for $(T,f)$. So the result is a consequence of upper semi-continuity of the entropy map $\nu\mapsto h_T(\nu)$.
%Suppose $\lim_{n\to\infty}f_n=f$ in $C^0(\mbbt)$. It suffices to show that $\limsup_{n\to\infty}\mclh(T,f_n)\le \mclh(T,f)$. For each $n$, choose an %arbitrary $\mu_n\in\mclm_{max}(T,f_n)$ with $\mclh(T,f_n)=h_{\mu_n}(T)$, whose existence follows from the upper semi-continuity of $\nu\mapsto %h_{\nu}(T)$ on $\mclm(T)$ (endowed with the usual weak-* topology) and compactness of $\mclm_{max}(T)\subset \mclm(T)$. By compactness, by choosing %subsequence if necessary, we may assume that $\mu_n\to\mu$ in $\mclm(T)$. Then the conclusion follows from the following two simple facts.
%  \begin{itemize}
%    \item $\mu\in \mclm_{max}(T,f)$. To see this, given $\nu\in \mclm(T)$, for each $n$, since $\int f_n\dif\nu\le \int f_n\dif\mu_n$, the following %holds:
%    \[
%    \int f\dif\nu  \le  \int (f-f_n)\dif\nu +\int (f_n-f)\dif\mu_n +\int f\dif\mu_n.
%    \]
%Letting $n\to \infty$, the first two terms on RHS are vanishing and the last one converges to  $\int f\dif\mu$.
%    \item $h_\mu(T)\ge\limsup_{n\to\infty}h_{\mu_n}(T)$. This is simply due to the upper semi-continuity of $\nu\mapsto h_{\nu}(T)$ on $\mclm(T)$.
%  \end{itemize}

%It remains to show that $\mclh$ is upper semi-continuous on $\mcle^1\times C^0(\mbbt)$.
%To this end, suppose
Now suppose that $(T_n,g_n)\to (T,f)$ in $\mcle^1\times C^0(\mathbb{R}/\mathbb{Z})$.
By \cite[Lemma 2]{JM08},
%and we have to show that $\limsup_{n\to\infty}\mclh(T_n,g_n)\le \mclh(T,f)$. To apply the result in the last paragraph, recall the following %well-known facts. Since $\lim_{n\to\infty}T_n=T$ in $\mcle^1$,
for each $n$ sufficiently large, there exists a homeomorphism $h_n:\mathbb{R}/\mathbb{Z}\to\mathbb{R}/\mathbb{Z}$ such that $T=h_n^{-1}\circ T_n\circ h_n$, and moreover, $\lim_{n\to\infty}\max_{x\in\mathbb{R}/\mathbb{Z} }d_{\mathbb{R}/\mathbb{Z}}(h_n(x),x)=0$, where $d_{\mathbb{R}/\mathbb{Z}}$ is the standard metric on $\mathbb{R}/\mathbb{Z}$.
%Since $T_n$ and $T$ are topologically conjugate via $h_n$,
Put $f_n:=g_n\circ h_n\in C^0(\mathbb{R}/\mathbb{Z})$. Then $\mclh(T,f_n)=\mclh(T_n,g_n)$ holds and $f_n\to f$ in $C^0(\mathbb{R}/\mathbb{Z})$.  Thus
\[
\limsup_{n\to\infty} \mclh(T_n,g_n) = \limsup_{n\to\infty} \mclh(T,f_n)\le \mclh(T,f).
\]
\end{proof}

\begin{proof}[Proof of Theorem~\ref{thm:Cr}]
We only need to consider the maximizing case. Let $T\in \mcle^\omega(\mathbb{R}/\mathbb{Z})$ and let $r\in\{0,1,2,\cdots,\infty\}$. Since the inclusion map from $C^r(\mathbb{R}/\mathbb{Z})$ to $C^0(\mathbb{R}/\mathbb{Z})$ is continuous, according to Proposition~\ref{prop:usc}, the function $f\mapsto \mclh(T,f)$ defined on $C^r(\mathbb{R}/\mathbb{Z})$ is upper semi-continuous. Therefore, $C_{ze}^r:=\{f\in C^r(\mathbb{R}/\mathbb{Z}): \mclh(T,f)=0\}$ is a $G_\delta$ subset of $C^r(\mathbb{R}/\mathbb{Z})$. On the other hand, by our Main Theorem, $C_{ze}^r\supset C_{ncc}^\omega$, where  $C_{ncc}^\omega$ denote the collection of functions in $C^\omega(\mathbb{R}/\mathbb{Z})$ that are not analytically cohomologous to constant. To complete the proof, it remains to show that $C_{ncc}^\omega$ is  dense in $C^r(\mathbb{R}/\mathbb{Z})$. To this end, let $C_{wcb}^r=\{f\in C^r(\mathbb{R}/\mathbb{Z}):\int f\dif\mu~\text{is independent of }\mu\in \mclm(T)\}$. Clearly, $C_{wcb}^r$ is a closed subset of $C^r(\mathbb{R}/\mathbb{Z})$ with empty interior, and $C^\omega(\mathbb{R}/\mathbb{Z})\setminus C_{wcb}^r\subset C_{ncc}^\omega$.
Since $C^\omega(\mathbb{R}/\mathbb{Z})$ is dense in $C^r(\mathbb{R}/\mathbb{Z})$, it follows that $C_{ncc}^\omega$ dense in $C^r(\mathbb{R}/\mathbb{Z})$, which completes the proof.
\end{proof}

We shall need the following well-known result for the proof of Theorems~\ref{thm:analytic Lyp} and ~\ref{thm:Cr Lyp}. See, for example, \cite[Proposition 28]{CLT01} for a more comprehensive version of this result under the $C^{1+\alpha}$ (and orientation-preserving) setting.

\begin{prop}\label{prop:Crcoho}
For each $T\in \mcle^r$, $r\in\{1,2,\ldots, \infty,\omega\}$, the following are equivalent:
\begin{enumerate}
\item [(i)] $\log |T'|$ is $C^0$ cohomologous to constant with respect to $T$;
\item [(ii)] $T$ is $C^r$ conjugate to the linear map $x\mapsto \deg(T) \cdot x$.
\end{enumerate}
\end{prop}

\begin{proof} The implication (ii) $\implies$ (i) is trivial. Let us show (i) $\implies$ (ii). Since $\log |T'|$ is $C^0$ cohomologous to constant, there exist $\psi\in C^0(\mathbb{R}/\mathbb{Z})$ and  $c\in \mathbb{R}\setminus\{0\}$ such that the following hold:
\begin{itemize}
  \item  $\psi>0$ and $\int_{\mathbb{R}/\mathbb{Z}}\psi(x)\dif x=1$;
  \item  $\psi\circ T\cdot T'=c\cdot\psi$.
\end{itemize}
The first item above implies that there exists a $C^1$-diffeomorphism $\phi:\mathbb{R}/\mathbb{Z}\to \mathbb{R}/\mathbb{Z}$ with $\phi'=\psi$. Then the second item above can be rewritten as $(\phi\circ T)'=c\cdot\phi'$. Integrating both sides over  $\mathbb{R}/\mathbb{Z}$ yields that $\deg(T)=c$. Therefore, $T$ is $C^1$ conjugate to the linear map $x\mapsto \deg(T) \cdot x$ via $\phi$.

It remains to show that $\phi$ is $C^r$ when $T\in\mcle^r$ for $r\in \{2,3,\ldots, \infty,\omega\}$. In this situation, $T$ admits a unique absolutely continuous invariant probability measure $\mu$, which has $C^{r-1}$ density. On the other hand, $\phi_*\mu$ is the unique absolutely continuous  invariant probability measure of $x\mapsto \deg(T) \cdot x$, which is exactly the standard Lebesgue measure. Thus $\phi$ is also $C^r$.
\end{proof}

\begin{proof}[Proof of Theorem~\ref{thm:analytic Lyp}] Let us only consider the maximizing case, i.e. $f=\log |T'|$, as the minimizing case is similar.
By the Main Theorem, either $f$ is analytically cohomologous to constant, or any Lyapunov maximizing measure has zero entropy. If the first case happens, then by Proposition~\ref{prop:Crcoho}, $T$ is $C^\omega$ conjugate to $x\mapsto \deg(T)\cdot x$.
\end{proof}
%The proof of Theorem~\ref{thm:Cr Lyp} is totally similar.

\begin{proof}[Proof of Theorem~\ref{thm:Cr Lyp}]
We only need to consider the maximizing case.  Let $r\in\{1,2,\cdots,\infty\}$. Since the map $T\mapsto (T,\log|T'|)$ from $\mcle^r$ to $\mcle^1\times C^0(\mathbb{R}/\mathbb{Z})$ is continuous, according to Proposition~\ref{prop:usc}, the function $T\mapsto \mclh(T,\log|T'|)$ defined on $\mcle^r$ is upper semi-continuous. Therefore, $\mcle_{ze}^r:=\{f\in \mcle^r: \mclh(T,\log|T'|)=0\}$ is a $G_\delta$ subset of $\mcle^r$. On the other hand, by Theorem~\ref{thm:analytic Lyp}, $\mcle_{ze}^r\supset \mcle_{ncc}^\omega$, where  $\mcle_{ncc}^\omega$ denote the collection of maps in $\mcle^\omega$ that are not analytically conjugate to linear map. To complete the proof, it remains to show that $\mcle_{ncc}^\omega$ is dense $\mcle^r$. To this end, let $\mcle_{wcb}^r=\{T\in \mcle^r:\int \log|T'|\dif\mu~\text{is independent of }\mu\in \mclm(T)\}$. Clearly, $\mcle_{wcb}^r$ is a closed subset of $\mcle^r$ with empty interior and $\mcle^\omega\setminus \mcle_{wcb}^r\subset \mcle_{ncc}^\omega$.
%and it is easy to see that  (\textcolor{red}{is it necessary to add some details here?}).
Since $\mcle^\omega$ is dense in $\mcle^r$, it follows that $\mcle_{ncc}^\omega$ dense in $\mcle^r$, which completes the proof.
\end{proof}

\bibliographystyle{plain}             %    other options: alpha, plain, abbrv ...

\bibliography{refer}

\end{document}